\documentclass[11pt,a4paper]{amsart}
\usepackage[T1]{fontenc}
\usepackage{lmodern}
\usepackage[margin=27mm]{geometry}
\usepackage{amsmath,amssymb,amsthm}
\usepackage[utf8]{inputenc}
\usepackage{tikz}
\usetikzlibrary{shapes.geometric}
\usepackage{booktabs}
\usepackage{xurl}
\usepackage{hyperref}
\hypersetup{draft=false,hidelinks,pdftitle={Optimal defective-chromatic separation for clique blowups},pdfauthor={Guillaume Lecomte}}
\usepackage[expansion=false]{microtype}
\usepackage{enumitem}
\usepackage{needspace}
\theoremstyle{plain}
\newtheorem{theorem}{Theorem}[section]
\newtheorem{proposition}[theorem]{Proposition}
\newtheorem{lemma}[theorem]{Lemma}
\newtheorem{corollary}[theorem]{Corollary}
\theoremstyle{definition}
\newtheorem{question}[theorem]{Question}
\numberwithin{equation}{section}
\DeclareMathOperator{\KG}{KG}
\date{}
\begin{document}
\title[Defective colouring of clique blowups]{Optimal defective--chromatic separation for clique blowups}
\author{Guillaume Lecomte}
\address{Independent researcher, Paris (France).}
\email{guillaume.lecomteexed@edu.executive.em-lyon.com}
\thanks{Corresponding author. ORCID: \href{https://orcid.org/0009-0003-3532-3379}{0009-0003-3532-3379}.}

\begin{abstract}
Guo, Kang and Zwaneveld conjectured that colouring the $(d+1)$-clique blowup of a graph with defect $d$ requires as many colours as colouring the graph properly. Norin and Steiner refuted this, proved that the two chromatic numbers differ by at most a factor of two, and asked for the exact constant; the best separation known was $11/10$. We prove that the constant is exactly two for every fixed defect $d\ge 2$. Every rational ratio in $[1,2)$ is attained, by graphs whose clique number and fractional chromatic number equal the defective parameter: the separation is invisible to the standard lower bounds. The engine is an exact transfer theorem: a weighted list-colouring obstruction spread over an arbitrary auxiliary graph yields prescribed defective and ordinary chromatic numbers; Kneser graphs then give sharpness. A 29-vertex graph with defective parameter four and chromatic number five answers Norin's Barbados question at the smallest possible defect and palette.
\end{abstract}

\subjclass[2020]{Primary 05C15; Secondary 05C72, 05C76}
\keywords{defective colouring, clique blowup, list colouring, chromatic separation, Kneser graph}
\maketitle

\section{Introduction}\label{sec:introduction}

A proper colouring assigns different colours to adjacent vertices.
A $d$-defective colouring relaxes this condition: each vertex may have
at most $d$ neighbours of its own colour. Write $\chi^d(H)$ for the
minimum number of colours in a $d$-defective colouring of $H$. For a graph $G$, its
$(d+1)$-clique blowup $G\boxtimes K_{d+1}$ replaces each vertex by a clique of order
$d+1$ and completely joins two such cliques whenever their original vertices
are adjacent. The parameter
\[
q_d(G):=\chi^d(G\boxtimes K_{d+1})
\]
will be compared with the ordinary chromatic number $\chi(G)$.
All graphs below are finite and simple. We write $\omega(G)$ and $\chi_f(G)$ for
the clique and fractional chromatic numbers, respectively.
The defect $d$ is a nonnegative integer, $[n]=\{1,\ldots,n\}$, and $N_G(v)$
denotes the set of neighbours of $v$ in $G$. We write $u\sim v$ for adjacency.

Defective colouring has long been studied in a variety of graph classes;
see Cowen, Cowen and Woodall~\cite{cowen} and the survey of Wood~\cite{wood}.
Here we fix the defect and study its effect on the colouring of clique blowups.

The choice of fibre size is critical: a monochromatic fibre of size $d+1$ is
allowed, but any additional neighbour of the same colour would violate the
defect constraint. Hence a proper colouring of $G$ lifts to a $d$-defective colouring
of the blowup, and $q_d(G)\le\chi(G)$. It is less obvious whether the reverse
inequality should hold. Motivated by spectral bounds for improper colouring,
Guo, Kang and Zwaneveld~\cite{guo} conjectured precisely that $q_d(G)=\chi(G)$.
Their supporting results included equality for $d\le 1$, perfect graphs and
graphs with $\chi(G)\le 4$, as well as fractional and clustered analogues.

Nevertheless, Norin and Steiner~\cite{norin} disproved the conjecture
for every $d\ge 2$, using a list-colouring obstruction of Bohman and
Holzman~\cite{bohman}. Their examples gave a ratio at least $30/29$
for infinitely many defects, while their independent-transversal argument
proved
\begin{equation}\label{eq:universal}
q_d(G)\le\chi(G)\le 2\chi^d(G\boxtimes K_d)\le 2q_d(G)\qquad(d\ge 1).
\end{equation}
Norin and Steiner~\cite[Problem 4]{norin} asked to determine the supremum
of $\chi(G)/q_d(G)$ over all nonempty graphs and all defects.
A subsequent construction at defect two improved the known lower
bound to $11/10$~\cite{lecomte}. A complementary question concerns the smallest
palette at which equality can fail. The proof of~\cite[Theorem 16]{guo}
shows that no separation is possible with at most three defective colours. Norin therefore
asked whether $q_d(G)\le 4$ always implies $\chi(G)\le 4$~\cite[Problem 29]{barbados}.

We connect these two questions by turning a local colouring obstruction
into a prescribed chromatic gap. We first construct a four-colour
counterexample. We then show how to place a copy of the obstruction at
each vertex of an auxiliary graph: the extra colours needed by the copies
are counted exactly by the chromatic number of that graph. Choosing Kneser
graphs as auxiliaries makes the resulting ratio tend to two.

\begin{theorem}[First separation]\label{thm:first}
There is a graph $W$ on $29$ vertices and $85$
edges such that
\[
\omega(W)=\chi_f(W)=q_2(W)=4,\qquad\chi(W)=5.
\]
Moreover, $q_{3k-1}(W)=4$ for every integer $k\ge 1$.
\end{theorem}

Thus Norin's question has a negative answer already with four defective
colours and defect two, the first possible values of both parameters.
Removing the four anchors from this example gives the weighted list-colouring
obstruction used in the general construction.

\begin{theorem}[Sharp separation and prescribed values]\label{thm:sharp}
For each integer
$d\ge 2$, there is an integer $b_d$ such that, for every $q\ge 2b_d$ and every integer
\[
q\le m\le 2q-2b_d+2,
\]
there is a graph $G$ satisfying
\begin{equation}\label{eq:prescribed}
\omega(G)=\chi_f(G)=q_d(G)=q,\qquad\chi(G)=m.
\end{equation}
One may take $b_d=4$ when $d\equiv 2\pmod 3$, and $b_d\le 2d^3+2d^2+d+3$ for
every $d\ge 2$. In particular,
\[
\sup_{G\ne\varnothing}\frac{\chi(G)}{q_d(G)}=2
\qquad\text{for every fixed }d\ge 2.
\]
\end{theorem}

Since the ratio equals one for $d=0,1$, Theorem~\ref{thm:sharp} also
gives a supremum of two over all nonempty finite graphs and all nonnegative
defects. This answers Problem 4 of Norin and Steiner~\cite{norin}.

Thus the separation spectrum
\[
\mathcal R_d=\{\chi(G)/q_d(G):G\text{ is a nonempty finite graph}\}
\]
contains every rational number in $[1,2)$ and no number above two. The
identity $\omega(G)=\chi_f(G)=q_d(G)$ in the examples shows that the separation
is purely integral: neither the clique number nor the fractional chromatic
number detects the additional colours needed by a proper colouring. The
possible attainment of the endpoint two is a different question, addressed at
the end of the paper.

The proof follows the same order as the underlying idea. Section~\ref{sec:obstruction}
constructs the four-colour obstruction. Section~\ref{sec:transfer} formulates the exact transfer
and combines it with Kneser graphs. Section~\ref{sec:endpoint} closes with the remaining
endpoint question.

\section{A four-colour obstruction}\label{sec:obstruction}
\subsection{Fibre multiplicities}

It is convenient to record how many times each colour occurs in a fibre.
These multiplicities describe both the small example and the later transfer. For $D=d+1$, a $d$-defective $q$-colouring of
$G\boxtimes K_D$ is equivalent to nonnegative integer multiplicities $m_i(v)$ satisfying
\begin{equation}\label{eq:multiplicities}
\sum_{i=1}^q m_i(v)=D,\qquad
m_i(v)>0\implies m_i(v)+\sum_{u\sim v}m_i(u)\le D.
\end{equation}
Indeed, a clone of colour $i$ above $v$ has exactly $m_i(v)-1+\sum_{u\sim v}m_i(u)$
neighbours of its colour. Multiplying all multiplicities by an integer $k\ge 1$ therefore gives a
$(kD-1)$-defective colouring of $G\boxtimes K_{kD}$ with the same palette.

We shall also use
\begin{equation}\label{eq:chain}
\omega(G)\le\chi_f(G)\le q_d(G)\le\chi(G).
\end{equation}
To explain the middle inequality, recall that a proper $b$-fold colouring
assigns a $b$-element colour set to each vertex, with disjoint sets on adjacent
vertices. Its minimum palette size is $\chi_b(G)$; see Stahl~\cite{stahl}.
The standard identity $\chi_f(G)=\inf_{b\ge 1}\chi_b(G)/b$ is discussed in
Scheinerman and Ullman~\cite{scheinerman}.
Now refine each colour class of a $d$-defective $q$-colouring of
$G\boxtimes K_D$ into at most $D$ independent sets, using disjoint palettes
for different original colours. This produces a proper $Dq$-colouring of
the blowup, equivalently a proper $D$-fold $Dq$-colouring of $G$.
Consequently $\chi_f(G)\le\chi_D(G)/D\le q$.

\subsection{An encoding of independent transversals}

Let $H$ have maximum degree at most two, and let $\mathcal P=\{P_1,\ldots,P_t\}$
partition $V(H)$ into triples. An independent transversal is an independent
set containing exactly one vertex of each $P_j$. Haxell's theorem~\cite{haxell2001}
guarantees an independent transversal when every part has size at least
$2\Delta(H)$; Szab\'o and Tardos~\cite{szabo} show that this threshold is sharp
for every positive maximum degree. Our triples at maximum degree two lie
just below this threshold. Choose maps
\[
a,b:V(H)\longrightarrow\{1,2,3\}
\]
such that $a|_{P_j}$ is a bijection for every $j$, and $b$ is a proper colouring of $H$.
Such a $b$ exists because $\Delta(H)\le 2$. Write $s=|\{v:a(v)\ne b(v)\}|$.

Construct $\Gamma=\Gamma(H,\mathcal P,a,b)$ in three steps.
\begin{enumerate}[label=(\roman*),leftmargin=*]
\item Take four anchors $z_1,z_2,z_3,z_4$ forming a clique and independent
centres $x_1,\ldots,x_t$, each with support $\{1,2,3\}$.
\item For each $v\in P_j$, create a private path $R_v$ and join its first
vertex to $x_j$. Label each path vertex by an ordered pair of distinct
colours; its support is the corresponding unordered pair. If $a(v)=b(v)$,
use the single pair $(a(v),4)$. Otherwise, let $c(v)$ be the remaining
colour in $\{1,2,3\}\setminus\{a(v),b(v)\}$ and use the four pairs
\begin{equation}\label{eq:pairs}
(a(v),4),\ (4,c(v)),\ (c(v),b(v)),\ (b(v),4).
\end{equation}
\item Write $y_v$ for the last vertex of $R_v$, and add $y_uy_v$ for each
$uv\in E(H)$. Join every non-anchor vertex to $z_i$ if and only if its
support does not contain $i$. Add no other edges.
\end{enumerate}

\begin{proposition}[Transversal encoding]\label{prop:encoding}
The graph $\Gamma$ has $4+4t+3s$ vertices
and satisfies
\[
\omega(\Gamma)=\chi_f(\Gamma)=q_2(\Gamma)=4,\qquad 4\le\chi(\Gamma)\le 5.
\]
Moreover, $\chi(\Gamma)=4$ if and only if $(H,\mathcal P)$ has an independent transversal.
The construction supplies both a $2$-defective four-colouring of $\Gamma\boxtimes K_3$ and a
proper five-colouring of $\Gamma$.
\end{proposition}

\begin{proof}
There are four anchors, $t$ centres and $3t+3s$ path vertices, giving
$4+4t+3s$ vertices in total.

To colour $\Gamma\boxtimes K_3$ with four colours, give all three vertices
above $z_i$ colour $i$. Above a centre, use $1,2,3$ once each. Above a path
vertex labelled $(r,u)$, use $r$ twice and $u$ once; call them its major
and minor colours. Table~\ref{tab:multiplicities} records
the multiplicity of each colour that occurs and its total multiplicity
in neighbouring fibres.
\begin{table}[htbp]
\centering
\small
\begin{tabular}{lcc}
\toprule
Vertex and colour & Multiplicity & Neighbour total\\
\midrule
Anchor: its own colour & $3$ & $0$\\
Centre: any supported colour & $1$ & $2$\\
Path vertex: major colour & $2$ & $1$\\
Nonterminal path vertex: minor colour & $1$ & $2$\\
Terminal: minor colour $4$ & $1$ & at most $2$\\
\bottomrule
\end{tabular}
\caption{Positive colour multiplicities and their totals in neighbouring fibres, for $D=3$.}
\label{tab:multiplicities}
\end{table}
Indeed, anchor adjacency forbids monochromatic edges to non-anchor
fibres. At a centre, each colour is major at exactly one first path
vertex, because $a|_{P_j}$ is bijective. Along a path, consecutive pairs
in~\eqref{eq:pairs} intersect only in the minor colour of the preceding
vertex, which is the major colour of the next. A terminal has at most
two terminal neighbours, and their major colours differ from its own
because $b$ is proper. Every row therefore satisfies~\eqref{eq:multiplicities}
with $D=3$. This gives $q_2(\Gamma)\le 4$; the anchor clique and~\eqref{eq:chain}
give $\omega(\Gamma)=\chi_f(\Gamma)=q_2(\Gamma)=4$.

On $\Gamma$ itself, a proper five-colouring is obtained by giving $z_i$
colour $i$, every centre colour $5$, and each path vertex its major colour.
Consecutive major colours differ, as do those at adjacent terminals.
The centre-to-path edges are proper because no path vertex uses colour $5$.
Path colours lie in their supports, and colour $5$ does not appear on an
anchor, so the anchor edges are proper as well.

It remains to determine when four colours suffice. Suppose first that
$\Gamma$ has a proper four-colouring, and rename colours so that $z_i$
has colour $i$.
The anchor edges force each non-anchor vertex to use its support. For each
$j$, select the unique $v\in P_j$ for which $a(v)$ is the colour of $x_j$.
The first vertex of $R_v$ must then use its minor colour; this forces every
subsequent vertex of $R_v$ to do the same. In the four-vertex case the
colours are $4,c(v),b(v),4$. Hence $y_v$ has colour $4$ in either case.
Two selected vertices cannot be adjacent in $H$, since their terminals
would then be adjacent and both coloured $4$. The selected vertices
therefore form an independent transversal.

Conversely, let $I$ be an independent transversal. Give $z_i$ colour $i$
and $x_j$ colour $a(v)$ for the unique $v\in I\cap P_j$. On $R_v$, use
all minor colours when $v\in I$, and all major colours otherwise.
The bijectivity of $a|_{P_j}$ makes the centre-to-path edges proper.
Along each path, consecutive major colours differ, and consecutive minor
colours differ. On a terminal edge, two minor colours cannot meet because
$I$ is independent; two major colours differ because $b$ is proper;
and a minor colour $4$ differs from every major colour. All colours lie
in their supports, so the anchor edges are also proper.
\end{proof}

\subsection{An explicit graph on 29 vertices}

The underlying graph $H$ below is the degree-two case of the extremal
family of Szab\'o and Tardos~\cite{szabo}. Structural results on degree-two
obstructions appear in Aharoni, Holzman, Howard and Spr\"ussel~\cite{aharoni}
and Haxell and Wdowinski~\cite{haxell}. The encoding above converts such a
transversal obstruction into the required separation of colouring parameters.

Let $H$ be the disjoint union of the three complete bipartite graphs with
bipartitions
\[
\{1,2\}\mid\{3,4\},\qquad\{5,6\}\mid\{7,8\},\qquad\{9,10\}\mid\{11,12\},
\]
and take
\begin{equation}\label{eq:parts}
\begin{gathered}
P_1=\{1,2,7\},\quad P_2=\{3,4,8\},\quad
P_3=\{5,9,10\},\quad P_4=\{6,11,12\}.
\end{gathered}
\end{equation}

There is no independent transversal. Indeed, a transversal cannot choose
from $\{1,2\}$ and $\{3,4\}$ simultaneously, so it must choose $7$ or $8$. It can then
choose neither $5$ nor $6$. It must therefore choose from both $\{9,10\}$ and
$\{11,12\}$, which again gives an edge. Figure~\ref{fig:cycles} depicts the three components
of $H$; the transversal parts are specified in~\eqref{eq:parts}.

\begin{figure}[htbp]
\centering
\begin{tikzpicture}[x=1.1cm,y=0.9cm,
  vertex/.style={circle,draw,fill=white,inner sep=1.5pt,minimum size=5mm}]
\foreach \offset/\a/\b/\c/\d in {0/1/2/3/4,2.7/5/6/7/8,5.4/9/10/11/12}{
  \begin{scope}[xshift=\offset cm]
  \coordinate (a) at (0,1);
  \coordinate (b) at (0,0);
  \coordinate (c) at (1,1);
  \coordinate (d) at (1,0);
  \draw (a)--(c) (a)--(d) (b)--(c) (b)--(d);
  \node[vertex] at (a) {$\a$};
  \node[vertex] at (b) {$\b$};
  \node[vertex] at (c) {$\c$};
  \node[vertex] at (d) {$\d$};
  \end{scope}
}
\end{tikzpicture}
\caption{The three disjoint four-cycles of the transversal obstruction. Each cycle is a $K_{2,2}$,
with bipartition given by its left and right columns; the transversal triples mix vertices
from different cycles.}
\label{fig:cycles}
\end{figure}

Use the labels in Table~\ref{tab:labels}.
\begin{table}[htbp]
\centering
\small
\begin{tabular}{c*{12}{r}}
\toprule
$v$&1&2&3&4&5&6&7&8&9&10&11&12\\
\midrule
$a(v)$&1&2&1&2&1&2&3&3&2&3&1&3\\
$b(v)$&1&1&2&2&1&2&3&3&2&2&1&3\\
\bottomrule
\end{tabular}
\caption{The two labellings defining the 29-vertex construction.}
\label{tab:labels}
\end{table}
The map $a$ is bijective on every part, and $b$ is proper on all three four-cycles.
They differ only at $2,3,10$.

Set $W=\Gamma(H,\mathcal P,a,b)$ for these data. Figure~\ref{fig:W} draws
the entire graph. Its four fans correspond to the parts $P_1,\ldots,P_4$;
the three longer branches are precisely $R_2,R_3,R_{10}$.

\begin{figure}[!htbp]
\centering
\begin{tikzpicture}[x=0.76cm,y=0.76cm,
  anchor edge/.style={draw=black!55,line width=0.4pt},
  terminal edge/.style={draw=black!80,densely dashed,line width=0.8pt},
  path edge/.style={draw=black!85,line width=0.85pt},
  wanchor/.style={diamond,draw=black!80,inner sep=1pt,minimum size=7mm},
  centre/.style={rectangle,rounded corners=1pt,draw=black!85,inner sep=2pt,minimum size=6mm},
  terminal/.style={circle,draw=black!80,inner sep=0pt,minimum size=5.7mm},
  internal/.style={circle,draw=black!75,inner sep=0pt,minimum size=3.2mm}]
\coordinate (z1) at (0.00000,7.70000);
\coordinate (z2) at (7.70000,0.00000);
\coordinate (z3) at (-0.00000,-7.70000);
\coordinate (z4) at (-7.70000,0.00000);
\coordinate (x1) at (-4.38406,4.38406);
\coordinate (r1s1) at (-0.99186,2.72511);
\coordinate (r2s1) at (-3.80070,3.80070);
\coordinate (r2s2) at (-3.21734,3.21734);
\coordinate (r2s3) at (-2.63397,2.63397);
\coordinate (r2s4) at (-2.05061,2.05061);
\coordinate (r7s1) at (-2.72511,0.99186);
\coordinate (x2) at (4.38406,4.38406);
\coordinate (r3s1) at (3.80070,3.80070);
\coordinate (r3s2) at (3.21734,3.21734);
\coordinate (r3s3) at (2.63397,2.63397);
\coordinate (r3s4) at (2.05061,2.05061);
\coordinate (r4s1) at (0.99186,2.72511);
\coordinate (r8s1) at (2.72511,0.99186);
\coordinate (x3) at (-4.38406,-4.38406);
\coordinate (r5s1) at (-2.72511,-0.99186);
\coordinate (r9s1) at (-2.05061,-2.05061);
\coordinate (r10s1) at (-3.53601,-3.96932);
\coordinate (r10s2) at (-2.68796,-3.55459);
\coordinate (r10s3) at (-1.83991,-3.13985);
\coordinate (r10s4) at (-0.99186,-2.72511);
\coordinate (x4) at (4.38406,-4.38406);
\coordinate (r6s1) at (2.72511,-0.99186);
\coordinate (r11s1) at (0.99186,-2.72511);
\coordinate (r12s1) at (2.05061,-2.05061);
\draw[anchor edge] (z1) .. controls (-2.77880,5.34547) and (-5.34547,2.77880) .. (z4);
\draw[anchor edge] (z1)--(z3);
\draw[anchor edge] (z1) .. controls (2.77880,5.34547) and (5.34547,2.77880) .. (z2);
\draw[anchor edge] (z2)--(z4);
\draw[anchor edge] (z2)--(z3);
\draw[anchor edge] (z3) .. controls (-2.70809,-5.27475) and (-5.27475,-2.70809) .. (z4);
\draw[anchor edge] (x1)--(z4);
\draw[anchor edge] (r1s1)--(z3);
\draw[anchor edge] (r1s1) .. controls (1.48660,0.48086) and (4.38388,-0.42751) .. (z2);
\draw[anchor edge] (r2s1) .. controls (-3.76814,-0.44079) and (-2.50124,-4.27435) .. (z3);
\draw[anchor edge] (r2s1)--(z1);
\draw[anchor edge] (r2s2)--(z1);
\draw[anchor edge] (r2s2) .. controls (0.95887,3.96740) and (4.59798,2.89495) .. (z2);
\draw[anchor edge] (r2s3) .. controls (1.57635,4.75994) and (5.02101,3.88195) .. (z2);
\draw[anchor edge] (r2s3)--(z4);
\draw[anchor edge] (r2s4) .. controls (-0.68206,-1.05553) and (0.00148,-4.30573) .. (z3);
\draw[anchor edge] (r2s4) .. controls (1.05553,0.68206) and (4.30573,-0.00148) .. (z2);
\draw[anchor edge] (r7s1) .. controls (-3.39174,3.86773) and (-2.48337,6.10378) .. (z1);
\draw[anchor edge] (r7s1)--(z2);
\draw[anchor edge] (x2) .. controls (-0.12142,4.23877) and (-4.14944,2.77742) .. (z4);
\draw[anchor edge] (r3s1)--(z2);
\draw[anchor edge] (r3s1) .. controls (3.76814,-0.44079) and (2.50124,-4.27435) .. (z3);
\draw[anchor edge] (r3s2)--(z2);
\draw[anchor edge] (r3s2)--(z1);
\draw[anchor edge] (r3s3) .. controls (-1.57635,4.75994) and (-5.02101,3.88195) .. (z4);
\draw[anchor edge] (r3s3)--(z1);
\draw[anchor edge] (r3s4)--(z1);
\draw[anchor edge] (r3s4) .. controls (0.68206,-1.05553) and (-0.00148,-4.30573) .. (z3);
\draw[anchor edge] (r4s1)--(z1);
\draw[anchor edge] (r4s1)--(z3);
\draw[anchor edge] (r8s1) .. controls (3.39174,3.86773) and (2.48337,6.10378) .. (z1);
\draw[anchor edge] (r8s1)--(z2);
\draw[anchor edge] (x3)--(z4);
\draw[anchor edge] (r5s1) .. controls (-2.37262,-3.45373) and (-1.46425,-5.68977) .. (z3);
\draw[anchor edge] (r5s1)--(z2);
\draw[anchor edge] (r9s1) .. controls (-0.68206,1.05553) and (0.00148,4.30573) .. (z1);
\draw[anchor edge] (r9s1) .. controls (-1.08508,-3.83138) and (-0.40154,-5.71451) .. (z3);
\draw[anchor edge] (r10s1) .. controls (-4.07999,0.44244) and (-2.90132,4.33222) .. (z1);
\draw[anchor edge] (r10s1) .. controls (0.57573,-3.68340) and (4.32107,-2.36029) .. (z2);
\draw[anchor edge] (r10s2) .. controls (1.19557,-3.59971) and (4.65823,-2.41485) .. (z2);
\draw[anchor edge] (r10s2)--(z3);
\draw[anchor edge] (r10s3)--(z4);
\draw[anchor edge] (r10s3)--(z3);
\draw[anchor edge] (r10s4)--(z1);
\draw[anchor edge] (r10s4)--(z3);
\draw[anchor edge] (x4) .. controls (-0.22374,-4.52079) and (-4.25176,-3.05943) .. (z4);
\draw[anchor edge] (r6s1) .. controls (0.48086,1.48660) and (-0.42751,4.38388) .. (z1);
\draw[anchor edge] (r6s1) .. controls (2.37262,-3.45373) and (1.46425,-5.68977) .. (z3);
\draw[anchor edge] (r11s1)--(z3);
\draw[anchor edge] (r11s1) .. controls (3.45373,-2.37262) and (5.68977,-1.46425) .. (z2);
\draw[anchor edge] (r12s1) .. controls (0.68206,1.05553) and (-0.00148,4.30573) .. (z1);
\draw[anchor edge] (r12s1)--(z2);
\draw[terminal edge] (r1s1)--(r4s1);
\draw[terminal edge] (r1s1)--(r3s4);
\draw[terminal edge] (r2s4)--(r4s1);
\draw[terminal edge] (r2s4)--(r3s4);
\draw[terminal edge] (r5s1)--(r8s1);
\draw[terminal edge] (r5s1)--(r7s1);
\draw[terminal edge] (r10s4)--(r12s1);
\draw[terminal edge] (r10s4)--(r11s1);
\draw[terminal edge] (r6s1)--(r8s1);
\draw[terminal edge] (r6s1)--(r7s1);
\draw[terminal edge] (r11s1)--(r9s1);
\draw[terminal edge] (r12s1)--(r9s1);
\draw[path edge] (r1s1)--(x1);
\draw[path edge] (r2s1)--(r2s2);
\draw[path edge] (r2s1)--(x1);
\draw[path edge] (r2s2)--(r2s3);
\draw[path edge] (r2s3)--(r2s4);
\draw[path edge] (r7s1)--(x1);
\draw[path edge] (r3s1)--(r3s2);
\draw[path edge] (r3s1)--(x2);
\draw[path edge] (r3s2)--(r3s3);
\draw[path edge] (r3s3)--(r3s4);
\draw[path edge] (r4s1)--(x2);
\draw[path edge] (r8s1)--(x2);
\draw[path edge] (r5s1)--(x3);
\draw[path edge] (r9s1)--(x3);
\draw[path edge] (r10s1)--(x3);
\draw[path edge] (r10s1)--(r10s2);
\draw[path edge] (r10s2)--(r10s3);
\draw[path edge] (r10s3)--(r10s4);
\draw[path edge] (r6s1)--(x4);
\draw[path edge] (r11s1)--(x4);
\draw[path edge] (r12s1)--(x4);
\node[wanchor,fill=white,text=black,font=\small] at (z1) {$z_{1}$};
\node[wanchor,fill=white,text=black,font=\small] at (z2) {$z_{2}$};
\node[wanchor,fill=white,text=black,font=\small] at (z3) {$z_{3}$};
\node[wanchor,fill=white,text=black,font=\small] at (z4) {$z_{4}$};
\node[centre,fill=black,text=white,font=\small] at (x1) {$x_{1}$};
\node[terminal,fill=white,text=black,font=\small] at (r1s1) {$1$};
\node[internal,fill=black,text=black,font=\small] at (r2s1) {};
\node[internal,fill=black,text=black,font=\small] at (r2s2) {};
\node[internal,fill=black,text=black,font=\small] at (r2s3) {};
\node[terminal,fill=white,text=black,font=\small] at (r2s4) {$2$};
\node[terminal,fill=white,text=black,font=\small] at (r7s1) {$7$};
\node[centre,fill=black,text=white,font=\small] at (x2) {$x_{2}$};
\node[internal,fill=black,text=black,font=\small] at (r3s1) {};
\node[internal,fill=black,text=black,font=\small] at (r3s2) {};
\node[internal,fill=black,text=black,font=\small] at (r3s3) {};
\node[terminal,fill=white,text=black,font=\small] at (r3s4) {$3$};
\node[terminal,fill=white,text=black,font=\small] at (r4s1) {$4$};
\node[terminal,fill=white,text=black,font=\small] at (r8s1) {$8$};
\node[centre,fill=black,text=white,font=\small] at (x3) {$x_{3}$};
\node[terminal,fill=white,text=black,font=\small] at (r5s1) {$5$};
\node[terminal,fill=white,text=black,font=\small] at (r9s1) {$9$};
\node[internal,fill=black,text=black,font=\small] at (r10s1) {};
\node[internal,fill=black,text=black,font=\small] at (r10s2) {};
\node[internal,fill=black,text=black,font=\small] at (r10s3) {};
\node[terminal,fill=white,text=black,font=\small] at (r10s4) {$10$};
\node[centre,fill=black,text=white,font=\small] at (x4) {$x_{4}$};
\node[terminal,fill=white,text=black,font=\small] at (r6s1) {$6$};
\node[terminal,fill=white,text=black,font=\small] at (r11s1) {$11$};
\node[terminal,fill=white,text=black,font=\small] at (r12s1) {$12$};
\node[font=\small] at (0,-8.55) {$29$ vertices\qquad $85$ edges};
\draw[path edge] (-5.9,-9.35)--(-4.9,-9.35);
\node[anchor=west,font=\scriptsize] at (-4.75,-9.35) {centre--path and internal path edges};
\draw[terminal edge] (-5.9,-9.95)--(-4.9,-9.95);
\node[anchor=west,font=\scriptsize] at (-4.75,-9.95) {terminal edges from $H$};
\draw[anchor edge] (-5.9,-10.55)--(-4.9,-10.55);
\node[anchor=west,font=\scriptsize] at (-4.75,-10.55) {edges incident with an anchor};
\end{tikzpicture}
\caption{The complete graph $W$, with $29$ vertices and $85$ edges. Diamonds are anchors, black squares are centres, numbered circles are terminals, and the nine small black circles are internal path vertices. Solid dark edges form the twelve branches; dashed edges reproduce $H$ on the terminals. Grey edges are incident with anchors, including the anchor clique. Crossings are not vertices.}
\label{fig:W}
\end{figure}

\Needspace{10\baselineskip}
\begin{proof}[Proof of Theorem~\ref{thm:first}]
Proposition~\ref{prop:encoding}, with $t=4$ and $s=3$, gives the asserted
parameter values, since the chosen partition has no independent transversal.
The graph consists of four anchors, four centres, nine one-vertex paths
and three four-vertex paths. The $85$ edges consist of $6$ anchor edges, $4$ centre-to-anchor edges, $42$
path-to-anchor edges, $12$ centre-to-path edges, $9$ internal path edges and $12$
terminal edges. The scaling observation following~\eqref{eq:multiplicities} gives $q_{3k-1}(W)\le 4$,
and the anchor clique gives equality.
\end{proof}

\section{From local obstructions to sharp separation}\label{sec:transfer}
\subsection{Pointed weighted obstructions}

After the anchors are removed, each vertex still has an allowed set of
colours, but these sets admit no proper colouring. We shall use such an
obstruction with a distinguished vertex whose deletion makes it colourable.
One new colour will then suffice to colour that vertex. The following
definition includes both our weighted example and the unweighted list
obstructions used by Norin and Steiner.

A pointed weighted $(d,b)$-instance consists of a nonempty graph $B$,
a vertex $p\in V(B)$, and nonnegative integer multiplicities $m_a(u)$,
indexed by $u\in V(B)$ and $a\in[b]$, where $b\ge 1$ is an integer. We write
the instance as $(B,m,p)$ and require
\begin{equation}\label{eq:weighted}
\sum_{a=1}^b m_a(u)=d+1,\qquad
m_a(u)>0\implies m_a(u)+\sum_{w\in N_B(u)}m_a(w)\le d+1.
\end{equation}
Put $L(u)=\{a:m_a(u)>0\}$. We require that $B$ have no proper $L$-colouring
and that $B-p$ have one. Here a proper $L$-colouring assigns each vertex $u$
a colour in $L(u)$, with different colours at adjacent vertices.
Unused colours of the ambient palette $[b]$ are
permitted. The unweighted special case has $m_a(u)=1$ on lists of size $d+1$;
then each colour in $L(u)$ belongs to the lists of at most $d$ neighbours of $u$.

In this unweighted case, non-colourability is precisely a counterexample to
Reed's list-colouring conjecture~\cite{reed}, disproved by Bohman and
Holzman~\cite{bohman}. This exact obstruction is compatible with the
asymptotic result of Reed and Sudakov~\cite{reedsudakov}: for every
$\varepsilon>0$ and all sufficiently large $d$, lists of size at least
$(1+\varepsilon)d$ suffice when the maximum vertex-colour degree is at most $d$.

\begin{lemma}\label{lem:instance}
A pointed weighted $(d,b)$-instance exists for every $d\ge 2$, with
$b\le 2d^3+2d^2+d+3$. For $d=3k-1$ with $k\ge 1$, the least possible palette size is exactly
four, and there is an instance with at most $25$ vertices.
\end{lemma}

\begin{proof}
Norin and Steiner state the Bohman--Holzman list obstruction in~\cite[Theorem 6]{norin}:
for every $d\ge 2$ it has colour-degree at most $d$, lists of size at
least $d+1$, and a palette consisting of exactly $2d^3+2d^2+d+3$ colours
(before shrinking lists). This is the quantitative form of the construction
in~\cite{bohman}. Shrink the lists to size $d+1$ (which cannot create a list-colouring), then
pass to a vertex-minimal induced subgraph which remains non-$L$-colourable;
any of its vertices can serve as $p$. Restricting to an induced subgraph
preserves the list sizes and can only decrease the colour-degrees, so the
resulting multiplicities still satisfy~\eqref{eq:weighted}.

For $d=3k-1$, remove the four anchors from the graph in Theorem~\ref{thm:first}.
Its $25$ remaining vertices, with the displayed multiplicities scaled by $k$,
satisfy~\eqref{eq:weighted}. They admit no proper colouring from their supports: such a
colouring would extend over the four anchors and contradict $\chi(W)=5$. Pass
again to a vertex-minimal induced non-list-colourable subgraph.
To see that fewer than four colours are impossible, interpret any
multiplicities satisfying~\eqref{eq:weighted} as a $d$-defective colouring
of $B\boxtimes K_{d+1}$. The proof of~\cite[Theorem 16]{guo} shows that a
colouring with at most three colours admits a proper colouring of $B$
chosen from the same supports. This would be a proper $L$-colouring,
a contradiction.
\end{proof}

\subsection{The exact transfer}

Fix a pointed weighted $(d,b)$-instance $(B,m,p)$ and a proper
$L$-colouring $\psi$ of $B-p$. Let $q\ge b$ be an integer, and let $J$ be
any finite graph with a proper $b$-fold $q$-colouring
\[
x\longmapsto T_x\in\binom{[q]}{b},\qquad
T_x\cap T_y=\varnothing\quad\text{whenever }xy\in E(J).
\]
Build $G=G(B,m;J,T)$ as follows.
\begin{enumerate}[label=(\roman*),leftmargin=*]
\item For each $x\in V(J)$, take a copy $B_x$ of $B$ and choose a bijection
$\varphi_x:[b]\to T_x$. The vertex corresponding to $u\in V(B)$ is denoted
by $u_x$, and its relabelled list is $L_x(u_x)=\varphi_x(L(u))$.
\item Completely join $B_x$ to $B_y$ for each edge $xy\in E(J)$.
\item Add a clique $Q=\{v_1,\ldots,v_q\}$ and join $v_a$ to $u_x$ if and
only if $a\notin L_x(u_x)$. Add no other edges.
\end{enumerate}
The clique $Q$ fixes the original $q$ colours. Each copy of $B$ will require
at least one additional colour, and the edges of $J$ determine which copies
can share it.

\begin{theorem}[Exact transfer]\label{thm:transfer}
For the graph just constructed,
\[
\omega(G)=\chi_f(G)=q_d(G)=q,\qquad\chi(G)=q+\chi(J).
\]
\end{theorem}

\begin{proof}
We first determine the defective parameter. In $G\boxtimes K_{d+1}$, colour the fibre above $u_x$ with multiplicities
$m_a(u)$, relabelled by $\varphi_x$, and colour the whole fibre above $v_a$
with $a$. Adjacent copies use disjoint palettes $T_x,T_y$, and the anchor
rule prevents monochromatic edges between $Q$ and the copies. Within a
copy,~\eqref{eq:weighted} bounds each monochromatic degree by $d$; above
an anchor the monochromatic degree is exactly $d$. Hence $q_d(G)\le q$.
Conversely, $Q\boxtimes K_{d+1}$ is a clique on $q(d+1)$ vertices, and a
$d$-defective colour class contains at most $d+1$ of them. Thus $q_d(G)=q$.

Now consider any proper colouring of $G$. Rename the $q$ colours on $Q$ as
$[q]$, and let $s$ be the number of additional colours. Each copy $B_x$
must use an additional colour: otherwise the anchor rule would force a
proper $L_x$-colouring, contrary to the choice of $B$. For each additional
colour, the copies in which it occurs form an independent set of $J$,
because adjacent copies are completely joined. These $s$ independent
sets cover $V(J)$. Assigning each vertex of $J$ one of the sets that
contains it gives a proper $s$-colouring, so $s\ge\chi(J)$.

For the reverse inequality, properly colour $J$ with $\chi(J)$ new colours,
all outside $[q]$. Give $p_x$ the new colour of $x$, give $u_x$ for $u\ne p$
the colour $\varphi_x(\psi(u))$, and give $v_a$ colour $a$.
Within each copy this is proper by the choice of $\psi$. Between adjacent
copies, base colours are disjoint and new colours differ; the two kinds
of colours cannot coincide. The anchor edges are proper by their
construction. Therefore $\chi(G)=q+\chi(J)$.

Finally, the anchor clique and~\eqref{eq:chain} give
$q\le\omega(G)\le\chi_f(G)\le q_d(G)=q$.
\end{proof}

The construction extends the join amplification of~\cite{norin}: the auxiliary graph
$J$ need not be complete. Its $b$-fold colouring controls the defective parameter,
while $\chi(J)$ controls the ordinary gap.

\subsection{Kneser amplification and prescribed chromatic numbers}

Let $b_d$ be the least integer $b$ for which a pointed weighted $(d,b)$-instance
exists. By Lemma~\ref{lem:instance}, this integer is well defined. For $q\ge 2b_d$, Lov\'asz's Kneser theorem states~\cite{lovasz}
\begin{equation}\label{eq:kneser}
\chi(\KG(q,b_d))=q-2b_d+2,
\end{equation}
where the vertices of $\KG(q,b_d)$ are the $b_d$-subsets of $[q]$ and edges join disjoint
subsets. Assigning each vertex its own subset is a proper $b_d$-fold $q$-colouring.
For a purely combinatorial proof of the same Kneser theorem, see
Matou\v{s}ek~\cite{matousek}. The use of Kneser graphs as target graphs for
multiple colourings is part of the framework developed by Stahl~\cite{stahl}.

\begin{proof}[Proof of Theorem~\ref{thm:sharp}]
If $m=q$, the complete graph $G=K_q$ satisfies~\eqref{eq:prescribed}.
Otherwise put $h=m-q$, so $1\le h\le q-2b_d+2$.
Starting with $\KG(q,b_d)$, delete vertices one at a time until a single
vertex remains. At each deletion the chromatic number decreases by
at most one. Since its initial value is $q-2b_d+2$ and its final value is
one, some induced subgraph $J$ has $\chi(J)=h$.
The natural $b_d$-fold $q$-colouring restricts to $J$, and
Theorem~\ref{thm:transfer} gives~\eqref{eq:prescribed}.

Taking the upper endpoint $m=2q-2b_d+2$ gives ratios $2-(2b_d-2)/q$ tending to two as
$q\to\infty$, while~\eqref{eq:universal} bounds every ratio by two.
Now write a rational number in $[1,2)$ as $1+a/c$, where $c\ge 1$ and
$0\le a<c$ are integers. Choose a positive integer $N$ large enough that
$Nc\ge 2b_d$ and $N(c-a)\ge 2b_d-2$.
Then $q=Nc$ and $m=N(c+a)$ satisfy $q\le m\le 2q-2b_d+2$, and the
constructed graph has ratio $m/q=1+a/c$.
\end{proof}

The four-colour base gives the following explicit interval.

\begin{corollary}[The four-colour base]\label{cor:four}
For every integer $d\ge 2$ with $d\equiv 2\pmod 3$, every
integer $q\ge 8$, and every integer $q\le m\le 2q-6$, there is a graph $G$ satisfying
$\omega(G)=\chi_f(G)=q_d(G)=q$ and $\chi(G)=m$. For every $q\ge 4$ there is also
such a graph with $\chi(G)=q+1$.
\end{corollary}

\begin{proof}
Use $b_d=4$ in Theorem~\ref{thm:sharp}. For the second assertion, use a one-vertex
auxiliary graph $J$ and any four-subset of $[q]$ in Theorem~\ref{thm:transfer}.
\end{proof}

\begin{corollary}[Closure of the separation spectrum]\label{cor:closure}
For every integer
$d\ge 2$, the separation spectrum $\mathcal R_d=\{\chi(G)/q_d(G):G\ne\varnothing\}$ satisfies
\[
\mathbb Q\cap[1,2)\subseteq\mathcal R_d\subseteq\mathbb Q\cap[1,2],
\qquad\overline{\mathcal R_d}=[1,2].
\]
\end{corollary}

\begin{proof}
The first inclusion follows from Theorem~\ref{thm:sharp}, and the second from
the universal bound~\eqref{eq:universal}. The rationals are dense in $[1,2]$.
\end{proof}

\section{The remaining endpoint}\label{sec:endpoint}

Theorem~\ref{thm:sharp} determines the sharp multiplicative constant and, more
strongly, realises every rational ratio strictly below it. The Kneser amplification
does not attain the endpoint: its maximal chromatic number is
$q-2b_d+2$, leaving an additive gap of $2b_d-2$ from $2q$ in the resulting
ordinary chromatic number. For $q\ge 2b_d$, this is optimal within the
transfer with a fixed $b_d$-colour base: a proper $b_d$-fold $q$-colouring
is a graph homomorphism $J\to\KG(q,b_d)$, so
$\chi(J)\le q-2b_d+2$.

\begin{question}[Attainment of the endpoint]\label{question:endpoint}
Does there exist a finite
nonempty graph $G$ and a defect $d\ge 2$ for which
\[
\chi(G)=2q_d(G)?
\]
\end{question}

If the answer is negative, integrality gives the sharper inequality
$\chi(G)\le 2q_d(G)-1$ for every nonempty finite graph. The optimal
multiplicative constant would still be two. The remaining question is
therefore whether equality can occur at a finite palette size.

\section*{Statements and Declarations}

\noindent\textbf{Funding.} This research received no external funding.

\medskip
\noindent\textbf{Author contributions.} The sole author is responsible for all aspects of this work.

\medskip
\noindent\textbf{Data availability.} No empirical datasets were generated or analysed.

\medskip
\noindent\textbf{Use of AI-assisted tools.} OpenAI ChatGPT (GPT-6;
\url{https://chatgpt.com}, accessed 4 November 2026) assisted with
English editing and LaTeX/TikZ preparation.

\end{document}